\documentclass[11pt]{article}
\pdfoutput=1
\usepackage[margin=2.5cm]{geometry}

\usepackage[utf8]{inputenc}
\usepackage{fourier}
\usepackage{microtype}
\usepackage{eurosym}

\usepackage{amsmath,amsfonts,amssymb,amsthm,mathtools}

\usepackage[style=alphabetic,
  safeinputenc,
  url=false,firstinits=true,
  sorting=nyt,
  maxnames=5]%
  {biblatex}
\renewbibmacro{in:}{
  \ifentrytype{article}{}{\printtext{\bibstring{in}\intitlepunct}}}
\DeclareFieldFormat
  [article,inbook,incollection,inproceedings,patent,unpublished,online]
  {title}{#1}
\DeclareFieldFormat
  [thesis]
  {title}{\mkbibemph{#1}}

\usepackage{hyperref}

\usepackage{pgffor}

\newtheorem{thm}{Theorem}[section]%
\newtheorem{prop}[thm]{Proposition}%
\newtheorem{lem}[thm]{Lemma}%
\theoremstyle{remark}
\newtheorem{rem}{Remark}%
\newcommand{\DeclareAlphabet}[2]{
  \foreach \x in {A,B,...,Z}{%
    \expandafter\xdef
    \csname #1\x\endcsname{%
      \noexpand#2{\x}}%
  }
}

\DeclareAlphabet{d}{\mathbb}  
\DeclareAlphabet{c}{\mathcal}
\DeclareAlphabet{r}{\mathrm}
\DeclareAlphabet{b}{\mathbf}

\newcommand{\ABS}[1]{{{\left\vert#1\right\vert}}} 
\newcommand{\BRA}[1]{{{\left\{#1\right\}}}} 
\newcommand{\PAR}[1]{{{\left(#1\right)}}} 
\newcommand{\prb}[2][]{\dP_{#1}\left[#2\right]}
\newcommand{\esp}[2][]{\dE_{#1}\left[#2\right]}

\DeclareMathOperator{\Hess}{Hess}
\DeclareMathOperator{\Vect}{Vect}

\DeclareMathOperator{\Cov}{Cov}
\DeclareMathOperator{\Var}{Var}
\DeclareMathOperator{\Ent}{Ent}

\newcommand{\eps}{\varepsilon}

\title{Functional inequalities for Gaussian convolutions of compactly
   supported measures: explicit bounds and dimension dependence}

\author{Jean-Baptiste~\textsc{Bardet}, Nathaël \textsc{Gozlan},
  Florent~\textsc{Malrieu} and Pierre-André~\textsc{Zitt}} %

\date{\today}

\begin{document}

\maketitle

\begin{abstract}
The aim of this paper is to establish various functional inequalities for the
convolution of a compactly supported measure and a standard Gaussian
distribution on $\dR^d$. We especially focus on getting good dependence of the
constants on the dimension. We prove that the Poincaré inequality holds with a
dimension-free bound. For the logarithmic Sobolev inequality, we improve the
best known results (Zimmermann, JFA 2013) by getting a bound that grows linearly 
with the dimension. We also establish transport-entropy inequalities for 
various transport costs. 

\medskip

\textbf{Keywords: } logarithmic Sobolev inequality, transport-entropy inequality, Poincaré
inequality

\textbf{MSC2010:} 60E15; 39B62; 26D10
\end{abstract}

\section{Introduction}

Poincaré or logarithmic Sobolev inequalities have been extensively 
studied in the past decades to quantify long time behavior of Markov 
processes or investigate the concentration of measure property, which 
plays a key role for example in the topic of large random matrices. 

We refer to \cite{Log-Sob,Ledoux-book,Royer,BGL14} for a comprehensive
introduction to this subject. Let us briefly recall some well-known
facts about these functional inequalities to motivate the present
study.

A probability measure $\nu$ on $\dR^d$ satisfies a Poincaré inequality
with constant $C$ if, for any smooth function~$f$ from $\dR^d$ to
$\dR$,
\[
\int_{\dR^d} f^2d\nu - \PAR{\int_{\dR^d} fd\nu}^2 \leq
C\int_{\dR^d}\ABS{\nabla f}^2d\nu.
\]
We denote by $C_P(\nu)$ the smallest constant such that this
inequality holds.

Similarly, $\nu$ satisfies a logarithmic Sobolev inequality with
constant $C$ if, for any smooth function $f$ from $\dR^d$ to $\dR$,
\[
\int_{\dR^d} f^2 \log (f^2)d\nu - \PAR{\int_{\dR^d} f^2d\nu}
\log \PAR{\int_{\dR^d} f^2d\nu} \leq C\int_{\dR^d}\ABS{\nabla
  f}^2d\nu,
\]
and we denote by $C_{LS}(\nu)$ the smallest constant such that this
inequality holds.

If $\nu$ is the Gaussian distribution $\cN_d(x,\Gamma)$ on $\dR^d$
with mean~$x$ and covariance matrix~$\Gamma$ then the values of these
optimal constants are known:
\[
C_P(\nu)=\frac{1}{2}C_{LS}(\mu)=\max \mathrm{Spec} (\Gamma).
\]
The Bakry-Émery criterion ensures that if $\nu$ has the density
$e^{-V}$ on $\dR^d$ and $\mathrm{Hess}(V)\geq \rho I_d$ then
\[
C_P(\nu)\leq \frac{1}{\rho} \quad \text{and}\quad C_{LS}(\mu)\leq
\frac{2}{\rho}.
\]
More generally, the inequality $2 C_P(\nu)\leq C_{LS}(\nu)$ always
holds.  These two functional inequalities do not hold if the support
of $\nu$ is not connected --- one can find a non constant function
whose gradient is zero $\nu$-almost surely.

\bigskip

The present paper focuses on the case when the probability measure
$\nu$ on $\dR^d$ is given by the convolution
$\mu\star \cN_d(0,\delta^2 I_d)$ where the support of $\mu$ is
included in the centered ball of $\dR^d$ with radius $R$.  This
question has been investigated recently in \cite{Zim,ZimJFA,WW13}; we
present here several improvements and related questions.

Let us fix some notation first.
\begin{itemize}
\item $X$ and $Z$ are two independent random variables with respective
  distribution $\mu$ and $\cN(0,I_d)$;
\item $\gamma_\delta$ is the density of the Gaussian measure
  $\cN_d(0, \delta^2 I_d)$;
\item $p$ is the density of the law $\mu\star\gamma_\delta$ of the
  random variable~$S=X+\delta Z$;
\item $C_d(\delta,R)$ is the supremum over all probability measures
  $\mu$ supported in the closed Euclidean ball $B_d(0,R)$ of the
  optimal constants in the logarithmic Sobolev inequality for
  $\mu\star\gamma_\delta$.
\end{itemize}
This notation is mainly consistent with \cite{Zim}, except that our
$\delta$ is the standard deviation of the Gaussian rather than its
variance, and we denote the dimension by $d$.

Zimmermann's results \cite{Zim,ZimJFA} may be summed up as follows.

\begin{thm}
  [Bounds on logarithmic Sobolev inequality constants,\cite{Zim}]
  \label{thm:zimm} 
  The convolution of a compactly supported measure and a Gaussian
  measure satisfies a logarithmic Sobolev inequality.  Moreover, there
  exist universal constants ${(K_i)}_{1\leq i\leq4}$ such that:
  \begin{itemize}
  \item In dimension $1$,
    \[
    C_1(\delta,R) \leq K_1 \frac{\delta^3 R}{4R^2+\delta^2}
    \exp\PAR{2\frac{R^2}{\delta^2}} + K_2(\delta+2R)^2.
    \]
    In particular in the low variance case $\delta \leq R$,
    \[
    C_1(\delta,R) \leq K_3 \frac{\delta^3}{R}
    \exp\PAR{2\frac{R^2}{\delta^2}}.
    \]
  \item In dimension $d$, $C_d(\delta,R)$ is finite. In the low
    variance case $\delta\leq R$, it satisfies:
    \[
    C_d(\delta,R) \leq K_4 R^2 \exp\PAR{ 20d + 5\frac{R^2}{\delta^2}}.
    \]
  \end{itemize}
\end{thm}

The proofs in \cite{Zim} rely on two main ideas. The one-dimensional
case is treated by explicit computations on Hardy-like criteria taken
from~\cite{BG}. In higher dimension the author applies the Lyapunov
function approach of~\cite{CGW10}.  The constants $K_i$ are explicit
but quite large (for example $K_4$ may be taken equal to $289$).  Let
us also mention the alternate approach of~\cite{zimTransport} in
dimension~$1$ by measure transportation, that unfortunately yields
even worse constants. In a related note, \cite{WW13} answer various
related questions on functional inequalities for convolutions, and
give many qualitative results under relaxed assumptions, both on the
support of $X$ and on the distribution of the mollifier $Z$, but
without exhibiting explicit constants.

We follow here the focus of~\cite{Zim} on quantitative estimates on
the constants and their dependence on the dimension~$d$.  Our first
result concerns the Poincaré inequality.

\begin{thm}[Dimension free Poincaré inequality]
  \label{thm:Poincare}
  If $\mu$ is supported in the closed Euclidean ball $B_d(0,R)$ then
  $\mu\star \gamma_\delta$ satisfies a Poincaré inequality and
  \[
  C_P(\mu\star\gamma_\delta) \leq \delta^2 \exp\PAR{4\frac{R^2}{\delta^2}}.
  \]
\end{thm}

The next result is an improvement on the bounds of
Theorem~\ref{thm:zimm}.
\begin{thm}[Bounds on the logarithmic Sobolev constants]\
  \label{thm:logSob}
  \begin{itemize}
  \item In the large variance case $\delta>R$, the logarithmic Sobolev
    constants are bounded uniformly in the dimension:
    \[
    C_d(\delta,R) \leq \frac{\delta^4}{\delta^2- R^2}.
    \]
  \item In dimension $1$, for any $\delta$, $R$,
    \[
    C_1(\delta,R) \leq 4\delta^2 \exp\PAR{ \frac{8}{\pi}
      \frac{R^2}{\delta^2}}.
    \]
  \item In the small variance case $\delta\leq R$, the logarithmic
    Sobolev constant admits the following dimension-dependent bound:
    \begin{equation}
      \label{eq:logSobDimDependant}
      C_d(\delta,R) \leq \PAR{K_1 d + K_2 \frac{R^2}{\delta^2}} R^2
      \exp\PAR{4\frac{R^2}{\delta^2}}
    \end{equation}
    where $K_1$, $K_2$ are universal constants.
  \end{itemize}
\end{thm}

The stronger bound in dimension $1$ is obtained as a corollary of a
bound that holds in any dimension (with a strong dependence on $d$).
Its proof uses a trick by Miclo to apply the classical Holley-Stroock
perturbation argument, and is much less technical than the ones
in~\cite{Zim,zimTransport}.

For the logarithmic Sobolev constant, the dependence in the dimension
drops from exponential to linear: this enhancement would translate
into weaker dependence assumptions in the applications to random
matrices considered in~\cite{Zim}.

\bigskip

In view of these results, it seems natural to conjecture as in
\cite{Zim} that $C_d(\delta,R)$ may admit a dimension free bound.  Let
us give some partial results in this direction.

The first is a dimension free bound for a transport-entropy
inequality.  We recall that if $k:\dR^d\times \dR^d \to \dR^+$ is a
cost function, then the optimal transport cost related to this $k$, is
defined, for all probability measures $\nu_1$ and $\nu_2$, by
\[
\cT_k(\nu_1,\nu_2) = \inf_{\pi} \int k(x,y)\,d\pi(x,y),
\]
where the infimum is taken over the set of all couplings $\pi$ between
$\nu_1$ and $\nu_2$.  Let $\cT_{2,4}$ and $\cT_{2}$ denote the
transportation costs associated to $(x,y)\mapsto \| x-y\|_4^2$ and
$(x,y)\mapsto |x-y|^2$ (here and in the whole paper, $|\cdot|$ denotes
the Euclidean norm).

\begin{thm}[Transportation-entropy inequality]
  \label{thm:transport}
  Let $\mu$ be a probability measure on $\dR^d$ supported in
  $B_d(0,R)$. The probability $\mu \star \gamma_\delta$ satisfies the
  following transport-entropy inequalities: for any probability
  measure $\nu$ on $\dR^d$,
  \begin{align*}
    \cT_{2,4}(\nu,\mu\star\gamma_\delta) 
    &\leq C(R,\delta)  H(\nu |\mu\star \gamma_\delta), \\
    \cT_{2}(\nu,\mu\star\gamma_\delta) 
    &\leq \sqrt{d} C(R,\delta)  H(\nu |\mu\star \gamma_\delta),
  \end{align*}
  where
  $C(R,\delta) = c'\delta^2 \PAR{1+
    \frac{R^2}{\delta^2}}\exp\PAR{\frac{4R^2}{\delta^2}}$
  for some universal constant $c'$.
\end{thm}
Let us remark that the factor $\sqrt{d}$ in this last result is better
than the linear factor $d$ that follows by deducing $\cT_{2}$ from the
logarithmic Sobolev inequality~\eqref{eq:logSobDimDependant} by
Otto-Villani's theorem (see \cite{OV00,BGL01}).

Finally, we are able to get bounds on the logarithmic Sobolev constant
in several restricted cases.

\begin{thm} [Partial results]
  \label{thm:partial}%
  \
  \begin{itemize}
  \item The quantity $C_d(\delta,R)$ may be bounded only in terms of
    $\delta$ and $R$ in the region $\delta > R/\sqrt{2}$.
  \item If $\mu$ is radially symmetric, then
    \[
    C_{LS}(\mu \star \gamma_\delta) \leq
    4\delta^2\exp\PAR{\frac{8}{\pi} \frac{R^2}{\delta^2}}.
    \]
  \item If $\mu$ is a uniform discrete probability measure on
    $N\geq 3$ points,
    \[
    C_{LS}(\mu\star\gamma_\delta) \leq \delta^2 + 3 \log(N)
    \delta^2\exp\PAR{4\frac{R^2}{\delta^2}}\,.
    \]
  \item The logarithmic Sobolev inequality restricted to log-convex
    functions holds with a constant that does not depend on the
    dimension.
  \end{itemize}
\end{thm}

To prove or disprove the conjecture, one is tempted to guess the
measure $\mu$ that leads to the worst logarithmic Sobolev constant. A
natural candidate, proposed in \cite[Example21]{Zim}, is the two-point
measure $1/2 (\delta_{Re_1} + \delta_{-Re_1})$ (where $e_1$ denotes
the first basis vector). Note that this candidate is easily seen to
satisfy a logarithmic Sobolev inequality with a bounded constant,
either by the bound on discrete measures or by a simple tensorization
argument of a one-dimensional convolution with a $(d-1)$-dimensional
Gaussian law. To build a counterexample one would have to consider
measures with a number of points that grows with the dimension.

\paragraph{Outline of the paper.}

The paper is organized as follows.  In Section~\ref{sec:perturbation}
we use the perturbation idea of Holley-Stroock, by rewriting the
potential of $\mu\star\gamma_\delta$ as a sum of a convex function and
a bounded perturbation, proving the first two items of
Theorem~\ref{thm:logSob}.  In Section~\ref{sec:mixing}, viewing
$\mu\star\gamma_\delta$ as a mixture of Gaussian measures we prove the
Poincaré and transportation inequalities (Theorems~\ref{thm:Poincare}
and~\ref{thm:transport}) and establish the bound for discrete measures
(third item of Theorem~\ref{thm:partial}). Theorem~\ref{thm:Poincare}
yields the final bound on logarithmic Sobolev constants (the third
item in Theorem~\ref{thm:logSob}) as an easy corollary.  The various
remaining results in Theorem~\ref{thm:partial} are proved in
Section~\ref{sec:partial}.

\section{Perturbation arguments}
\label{sec:perturbation}
\subsection{Large variance}
The density $p$ of $\mu\star\gamma_\delta$ is given explicitly by :
\[
p(z)=\int_{\dR^d} \frac{1}{(2\pi
  \delta^2)^{d/2}}\exp\PAR{-\frac{\ABS{z-x}^2}{2\delta^2}} \mu(dx)
=\frac{1}{(2\pi\delta^2)^{d/2}} \exp\PAR{ - \PAR{
    \frac{\ABS{z}^2}{2\delta^2} + W_\delta(z)} }
\]
where
\begin{align*}
  W_\delta(z)
  &= - \log\int_{\dR^d} \exp\PAR{\frac{z\cdot x}{\delta^2} 
    -\frac{\ABS{x}^2}{2\delta^2}}\,\mu(dx) \\
  &= - \log\int_{\dR^d} \exp\PAR{ \frac{z\cdot x}{\delta^2}}\,\nu(dx) 
    -\log C_\nu
\end{align*}
for $C_\nu = \int_{\dR^d} \exp(-\ABS{x}^2/(2\delta^2)) \mu(dx)$ and
$\nu(dx) = C_\nu^{-1} \exp(-\ABS{x}^2/(2\delta^2)) \mu(dx)$.  Let us
compute the Hessian of $(-W_\delta)$:
\[
\partial_{z_i}(- W_\delta)(z) = \frac{1}{\delta^2} \dE(\tilde X_i)
\quad\text{and}\quad
\partial^2_{z_i z_j}(- W_\delta)(z) = \frac{1}{\delta^4} \Cov(\tilde
X_i, \tilde X_j)
\]
where the distribution of $\tilde X$ is proportional to
$\exp(z\cdot x) d\nu$.  Therefore, for any unit vector $v$,
\[
0\leq \Hess (-W_\delta) v \cdot v \leq \frac{1}{\delta^4} \Var( v\cdot
\tilde X).
\]
Since $v\cdot \tilde X$ lives in $[-R,R]$, its variance is bounded by
$R^2$, so
\[
\Hess (- \log(p)) \geq \PAR{\frac{1}{\delta^2} - \frac{R^2}{\delta^4}}
I_d.
\]
\begin{rem}
  This bound is slightly better than the one given in~\cite{Zim} where
  the variance of $v\cdot\tilde X$ is bounded by $2R^2$.
\end{rem}

In particular, if $\delta>R$, $p$ is log-concave and the Bakry-Émery
criterion yields:
\[
  C_{LS}(\mu\star\gamma_\delta) \leq \frac{\delta^4}{\delta^2- R^2}.
\]
This proves the first item in Theorem~\ref{thm:logSob}.

\subsection{A perturbation argument}
\label{sec:Miclo}

It turns out we can get a (dimension dependent) bound on the
logarithmic Sobolev constant with a very short proof, using the
following trick to decompose the logarithm of the density $p$ as a sum
of a convex function and a bounded perturbation.

Let $a_d = \esp{\ABS{Z}}$ be the expected value of the norm of a
standard Gaussian random variable $Z$ in dimension $d$. Note that
$a_d$ has an explicit expression (we will use below that
$a_1=\sqrt{2/\pi}$) and is in any case smaller than $\sqrt{d}$.

\begin{lem}[Miclo's trick, \cite{Ledoux-revisited,Royer}]
  \label{lem:Royer}
  Suppose the function $W:\dR^d\to\dR$ may be written as $W=W_c+W_l$
  where $\Hess(W_c)\geq \rho I_d$, and $W_l$ is $l$-Lipschitz with
  respect to the Euclidean distance.

  Then for any $\sigma > 0$, one can write $W$ as a sum $U_c + U_b$
  where $\Hess(U_c) \geq \PAR{\rho - \frac{la_1}{\sigma}}I_d$ and
  $U_b$ is bounded by $l\sigma a_d$.

  In particular the measure $Z_W^{-1} \exp(-W)$ satisfies a
  logarithmic Sobolev inequality and
  \[
  C_{LS}\PAR{Z_W^{-1}\exp(-W)} \leq \frac{4}{\rho} \exp\PAR{
    \frac{4}{\rho} l^2 a_1a_d}\,.
  \]
\end{lem}

By way of comparison, it is known (see \cite{AidShi94,Aid98}) that if
$\mu_0 = \exp(-V_0)dx$ satisfies a logarithmic Sobolev inequality,
then $\nu = \exp(-V)dx $ satisfies a defective logarithmic Sobolev
inequality, as soon as the gradient $\nabla(V-V_0)$ satisfies some
exponential integrability condition.  This defective inequality can be
used together with the Poincaré inequality to obtain the logarithmic
Sobolev inequality. This strategy is used in \cite{WW13} (see in
particular \cite[Lemma 2.3]{WW13} for a precise statement of the
perturbation result). It is more general, since it only supposes a
logarithmic Sobolev inequality for the unperturbed measure, and
replaces a boundedness assumption by an integrability condition.  The
trade-off is that the constants are not explicit.

Since the statement of Lemma~\ref{lem:Royer} in
\cite{Ledoux-revisited,Royer} contains a typo in the convexity bound,
let us provide a detailed proof.

\begin{proof}
  Let $\sigma>0$ and $ U_\sigma$ be the following regularized version
  of $W_l$: $U_\sigma(x) = \esp{W_l(x+\sigma Z)}$, where $Z$ is a
  standard $d$-dimensional Gaussian random variable. Let
  $U_c = W_c + U_\sigma$ and $U_b= W_l - U_\sigma$. Since $W_l$ is
  $l$-Lipschitz,
  \[
  \ABS{U_b(x)} = \ABS{ \esp{W_l(x) - W_l(x+\sigma Z)}} \leq l\sigma
  \esp{\ABS{Z}} \leq l\sigma a_d\,.
  \]
  Therefore $U_b$ is bounded.
  
  We now turn to the convexity bound. It is enough to prove that, for
  any unit vector $v$ in $\dR^d$,
  $\Hess U_\sigma v \cdot v \leq \frac{lc_1}{\sigma}$.  First we
  compute the derivatives of $U_\sigma$:
  \begin{align*}
    \partial_i U_\sigma(x)
                  &= (2\pi \sigma^2)^{-d/2}\frac{1}{\sigma^2} 
                    \int_{\dR^d} W_l(y)(y_i-x_i) 
                    \exp\PAR{ - \frac{\ABS{x-y}^2}{2\sigma^2}} dy \\
                  &= (2\pi \sigma^2)^{-d/2}\frac{1}{\sigma^2} 
                    \int_{\dR^d} W_l(x+z) z_i 
                    \exp\PAR{ - \frac{\ABS{z}^2}{2\sigma^2}} dz \\
    \partial_{ij} U_\sigma(x)
                  &= (2\pi \sigma^2)^{-d/2}\frac{1}{\sigma^2} 
                    \int_{\dR^d} \partial_j W_l(x+z) z_i 
                    \exp\PAR{ - \frac{\ABS{z}^2}{2\sigma^2}} dz. 
  \end{align*}
  Now,
  \[
  \Hess U_\sigma v \cdot v= (2\pi \sigma^2)^{-d/2}\frac{1}{\sigma^2}
  \int_{\dR^d} (v\cdot\nabla W_l(x+z)) (v\cdot z)\exp\PAR{ -
    \frac{\ABS{z}^2}{2\sigma^2}} dz\,.
  \]
  Since $W_l$ is $l$-Lipschitz,
  \begin{align*}
    \ABS{\Hess U_\sigma v \cdot v} &\leq 
                                     (2\pi \sigma^2)^{-d/2}\frac{l}{\sigma^2} 
                                     \int_{\dR^d}  \ABS{v\cdot z} 
                                     \exp\PAR{ - \frac{\ABS{z}^2}{2\sigma^2}} dz. 
  \end{align*}
  By rotation invariance of the standard Gaussian distribution, we get
  \[
  \ABS{\Hess U_\sigma v \cdot v} \leq \frac{l}{\sigma^2} \esp{ \sigma
    \ABS{Z_1}} \leq \frac{l a_1}{\sigma}\,.
  \]
  This implies that
  $\Hess(W_c + U_\sigma) \geq (\rho-\frac{la_1}{\sigma})I_d$, as
  claimed.

  The final claim is a direct consequence of the obtained
  decomposition with $\sigma=2la_1/\rho$, the Holley--Stroock
  perturbation Lemma and the Bakry--Émery criterion (see
  \cite{Royer}).
\end{proof}

Let us now use this lemma to prove the one-dimensional bound in
Theorem~\ref{thm:logSob}. Write $-\log(p)$ as
\[
- \log (p(z)) = \PAR{\frac{\ABS{z}^2}{2\delta^2} + \frac{d}{2}
  \log(2\pi\delta^2)} + W_\delta(z).
\]
The first term is $\delta^{-2}$-convex. Since
\[
\nabla W_\delta(z)= -\frac{1}{\delta^2} \frac{\int_{\dR^d}\!  x
  \exp\PAR{\frac{z\cdot x}{\delta^2}} \,\nu(dx)} {\int_{\dR^d}\!
  \exp\PAR{\frac{z\cdot x}{\delta^2}} \,\nu(dx)},
\]
and $\nu(B_d(0,R))=1$, $W_\delta$ is $R/\delta^2$-Lipschitz on
$\dR^d$.  Lemma~\ref{lem:Royer} then yields
\[
C_{LS}(\mu\star\gamma_\delta) \leq 4\delta^2 \exp\PAR{ 4a_1a_d
  R^2\delta^{-2}}\,.
\]
This gives a first dimension dependent bound that is not comparable to
the one from Theorem~\ref{thm:zimm}. In dimension $1$, since
$a_1 = \sqrt{2/\pi}$, we get the bound claimed in the second item of
Theorem~\ref{thm:logSob}.

\section{Mixture arguments}
\label{sec:mixing}
\subsection{Poincaré inequality}
\label{sec:Poincare}
In this section we denote by $\gamma_{x,\delta}$ the distribution
$\cN_d(x,\delta^2I_d)$. Recall that
$\mu \star \gamma_\delta = \int_{\dR^d} \gamma_{x,\delta} d\mu(x)$.
The variance of a function $f$ under the mixture
$\mu \star \gamma_\delta$ can be classically decomposed as
\begin{align*}
  \Var_{\mu \star \gamma_\delta}(f) 
  &= \int_{\dR^d} \Var_{\gamma_{x,\delta}} (f) d\mu(x)  
    + \Var_{\mu} \PAR{x\mapsto \int f d\gamma_{x,\delta}} \\
  &= A + B.
\end{align*}
Since $\gamma_{x,\delta}$ satisfies the Poincaré inequality with
constant $\delta^2$, the first term $A$ is bounded by
\[
\delta^2 \int_{\dR^d} \int_{\dR^d} \ABS{\nabla f}^2 d\gamma_{x,\delta}
d\mu(x) = \delta^2 \int_{\dR^d} \ABS{\nabla f}^2
d(\mu\star\gamma_\delta)\,.
\]
For the second term $B$ let
$g:x\mapsto \int_{\dR^d} f d\gamma_{x,\delta}$. Duplicating variables
yields
\[
B = \frac{1}{2} \iint_{\dR^d\times\dR^d} (g(x) - g(y))^2 d\mu(x)
d\mu(y).
\]
Now
\begin{align*}
  (g(x) - g(y))^2 &= \PAR{ \int_{\dR^d} f d\gamma_{x,\delta} 
                    - \int_{\dR^d} f d\gamma_{y,\delta} }^2 
                    = \PAR{ \int_{\dR^d} f\PAR{1 
                    - \frac{d\gamma_{y,\delta}}{d\gamma_{x,\delta}}} 
                    d\gamma_{x,\delta} }^2 \\
                  &= \PAR{ \Cov_{\gamma_{x,\delta}} 
                    \PAR{f,\PAR{1 - \frac{d\gamma_{y,\delta}}
                    {d\gamma_{x,\delta}}}}}^2 \\
                  &\leq \Var_{\gamma_{x,\delta}}(f) \Var_{\gamma_{x,\delta}} 
                    \PAR{ 1 - \frac{d\gamma_{y,\delta}}{d\gamma_{x,\delta}}}
\end{align*}
by Cauchy-Schwarz inequality. For the first factor we reapply the
Poincaré inequality for the Gaussian measure~$\gamma_{x,\delta}$. The
second factor is the $\chi^2$ divergence between the Gaussian
distributions $\gamma_{x,\delta}$ and $\gamma_{y,\delta}$.  An easy
computation shows that this divergence is
$\PAR{\exp\PAR{\ABS{x-y}^2/\delta^2} - 1}$; since $\ABS{x-y}$ is
bounded by $2R$, we get
\[
(g(x) - g(y))^2 \leq \delta^2\PAR{\exp\PAR{4R^2/\delta^2}-1}
\int_{\dR^d}\ABS{\nabla f}^2 d\gamma_{x,\delta}.
\]
Reintegrating with respect to $\mu$ yields
\[
B \leq \delta^2 \PAR{\exp(4R^2/\delta^2)-1} \int_{\dR^d} \ABS{\nabla
  f}^2 d(\mu\star\gamma_\delta),
\]
so that the measure $\mu\star\gamma_\delta$ satisfies a Poincaré
inequality with a constant
\[
C_P(\mu\star\gamma_\delta)\leq \delta^2 \exp(4R^2/\delta^2)\,.
\]

\subsection{A mild dependence on \texorpdfstring{$d$}{d} for
  logarithmic Sobolev constants via Lyapunov functions}

The proof of the logarithmic Sobolev inequality in dimension greater
than $1$ in \cite{Zim} is based on a criterion from~\cite{CGW10}.
This criterion uses a Lyapunov function approach to prove a so-called
defective logarithmic Sobolev inequality, which can then be
strengthened using the Poincaré inequality.  In~\cite{Zim}, this
Poincaré inequality is itself obtained by Lyapunov criteria, with
constants depending exponentially on the dimension.  Simply plugging
our dimension-free Poincaré inequality in the argument of \cite{CGW10}
gives a much better bound.

Let us first recall the criterion, in the form used in \cite{Zim},
where the constants are explicitly written.
\begin{thm}[Logarithmic Sobolev inequality via Lyapunov functions,
  \cite{CGW10}]
  Suppose that $V$ satisfies
  \begin{align*}
    \Hess(V) \geq -K I_d
  \end{align*}
  with $K \geq 0$, and there exists a ``Lyapunov function'', that is,
  a function $W\geq 1$ such that
  \begin{equation}
    \label{eq:lyap}
    \Delta W - \langle \nabla V, \nabla W \rangle \leq (b-c\ABS{x}^2) W
  \end{equation}
  for some positive constants $b$, $c$.

  Suppose that $\nu=Z_V^{-1}\exp(-V)dx$ satisfies a Poincaré
  inequality with constant $C_P(\nu)$.  Let $A$ and $B$ be defined by
  \begin{align*}
    A &= \frac{2}{c} \PAR{\eps^{-1} + K/2} + \eps,  \\
    B &= \frac{2}{c} \PAR{\eps^{-1} + K/2}\PAR{b+c 
        \int_{\dR^d} \ABS{x}^2 d\nu(x)}.
  \end{align*}

  Then $\nu$ satisfies a logarithmic Sobolev inequality and
  $C_{LS}(\nu)\leq A+ (B+2)C_P(\nu)$.
\end{thm}

Zimmermann proves in \cite{Zim} that \eqref{eq:lyap} holds with
$b = d/(8\delta^2) + R^2/(32\delta^4)$ and $c = \frac{1}{64\delta^4}$
for the function $W(x)=\exp\PAR{\frac{1}{64\delta^4}}$.  Using the
bound $K\leq R^2/\delta^4$ and choosing $\eps=2/K$, this proves that,
for $\delta\leq R$, thanks to the bound on the Poincaré constant,
\[
C_{LS}(\mu \star \gamma_\delta) \leq \PAR{K_1 d + K_2
  \frac{R^2}{\delta^2}} R^2 \exp\PAR{4\frac{R^2}{\delta^2}}
\]
for some universal constants $K_1$, $K_2$, which is the general bound
announced in Theorem~\ref{thm:logSob}.

\subsection{A bound for uniform discrete measures}
Suppose in this section that $\mu$ is a uniform probability measure on
$N$ points in $B_d(0,R)$:
\[
\mu = \frac{1}{N} \sum_{i=1}^N \delta_{x_i}.
\]
The distribution of $S = X+Z_\delta$ is a mixture of $N$ Gaussian laws
with respective means $x_i$ and common covariance matrix
$\delta^2I_d$. Poincaré and logarithmic Sobolev inequalities for
mixtures of two measures have been studied by Chafa\"i and Malrieu in
\cite{CM10}; Schlichting and Menz \cite{Sch,MS14} have used and
generalized their results to prove Eyring-Kramers formul\ae.  The
decomposition of the variance used in Section~\ref{sec:Poincare} has
the following analogue for entropies:
\begin{equation}
  \label{eq:dec_entropy}
  \Ent_{\mu\star\gamma_\delta}\PAR{f^2} = \int_{\dR^d} \Ent_{\gamma_{x,\delta}}\PAR{f^2} 
  d\mu(x) + \Ent_\mu \PAR{x\mapsto \int_{\dR^d} f^2 d\gamma_{x,\delta}}.
\end{equation}
To bound the second term, we use the following result, that is
essentially a consequence of the discrete logarithmic Sobolev
inequality for the complete graph proved by Diaconis and Saloff-Coste
in \cite{DSC96}.
\begin{thm}
  [Upper bound for the entropy when $\mu$ is discrete, \cite{Sch}]
  \label{th:Sch}
  Let $\mu = \sum_{i=1}^N Z_i \mu_i$ be a finite mixture of
  measures. Let $Z^\star = \min_{1\leq i\leq N} (Z_i)$. Then for any
  $f$,
  \[
  \Ent_\mu\PAR{i \mapsto \int_{\dR^d} f^2d\mu_i} \leq
  \frac{1}{\Lambda(Z_\star, 1 - Z_\star)}
  \PAR{ \sum_{i=1}^N Z_i \Var_{\mu_i}(f) + \Var_\mu\PAR{i\mapsto
      \int_{\dR^d} f d\mu_i}},
  \]
  where $\Lambda(p,q) = (p-q)/(\log p - \log q)$.
\end{thm}
\begin{proof}
  This follows from \cite[Corollary 2.18]{Sch}, using the result from
  \cite{DSC96} instead of the alternate \cite[Lemma 2.13]{Sch}.
\end{proof}

Coming back to the decomposition~\eqref{eq:dec_entropy}, we can use
the Gaussian logarithmic Sobolev inequality on the first term and
Theorem~\ref{th:Sch} on the second term to get:

\[
\Ent_{\mu\star\gamma_\delta}(f^2) \leq 2\delta^2 \int_{\dR^d}
\ABS{\nabla f}^2 d\mu(x) + \frac{1}{\Lambda(1/N, (N-1)/N)}
\PAR{ \frac{1}{N} \sum_{i=1}^N \Var_{\gamma_{\delta,x_i}}(f) +
  \Var_\mu\PAR{i\mapsto \int_{\dR^d} f d\gamma_{\delta,x_i}}}.
\]
The last bracket is the variance $\Var_{\mu\star\gamma_\delta}(f)$,
which is bounded thanks to the Poincaré inequality.  Since
\( \frac{1}{\Lambda(p,1-p)} \leq \frac{ \log(1/p)}{1-2p} \),
we finally get
\[
C_{LS}(\mu\star\gamma_\delta)\leq 2\delta^2 + 3 \log(N)
\delta^2\exp(4R^2/\delta^2).
\]

\subsection{Dimension free transport-entropy inequality for the
  \texorpdfstring{$\ell^4$}{l4} norm}
We now adapt the arguments of Section~\ref{sec:Poincare} to prove that
the measure $\mu\star\gamma_\delta$ satisfies a transport-entropy
inequality with a constant depending only on $R$ and $\delta$.  It is
more convenient in this section to state and prove all intermediate
results for $\delta=1$.  In the final result we come back to the
general case by an immediate scaling argument.

The first step is to establish a weighted version of the Poincaré
inequality.

\begin{lem}[Weighted Poincaré inequality for Gaussian measures]
  \label{lem:wP-Gauss}
  For all $x \in \dR^d$, the Gaussian measure $\gamma_{x,1}$ satisfies
  the following weighted Poincaré inequality: for all $\mathcal{C}^1$
  function $f$,
  \[
  \Var_{\gamma_{x,1}}(f) \leq c (1+|x|^2) \int_{\dR^d}\sum_{i=1}^d
  \frac{1}{1+u_i^2} (\partial_if(u))^2\,d\gamma_{x,1}(u)\,,
  \]
  where $c$ is a positive universal constant.
\end{lem}
\begin{proof}
  Let us first establish the result for the standard Gaussian
  distribution $\gamma= \cN(0,1)$ in dimension $d=1$. According to the
  well known Muckenhoupt criterion for Hardy type inequalities (see
  \textit{e.g.} \cite[Theorem 6.2.1]{Log-Sob}), the inequality
  \[
  \int_0^\infty \PAR{f(u)-f(0)}^2\,d\gamma(u) \leq c \int_0^\infty
  \frac{1}{1+u^2} f'(u)^2\, d\gamma(u)
  \]
  holds for all $\mathcal{C}^1$ function $f: [0,\infty) \to \dR$, with
  the constant
  \[
  c = \sup_{y\geq 0} \int_y^\infty e^{-u^2/2} \,du\int_0^y
  (1+u^2)e^{u^2/2}\,du <\infty.
  \]
  Similarly, for any $\mathcal{C}^1$ function $f$ on $(-\infty,0]$, it
  holds
  \[
  \int_{-\infty}^0 \PAR{f(u)-f(0)}^2\,d\gamma(u) \leq c
  \int_{-\infty}^0 \frac{1}{1+u^2} f'(u)^2\, d\gamma(u)\,.
  \]
  Therefore, if $f$ is now $\mathcal{C}^1$ function on $\dR$, one has
  \[
  \Var_{\gamma}(f) \leq \int_{\dR} (f(u)-f(0))^2\,d\gamma(u) \leq c
  \int_{\dR} \frac{1}{1+u^2} f'(u)^2\, d\gamma(u)\,.
  \]
  Applying this inequality to $f(u) = g(x + u)$, $u \in \dR$, yields
  \[
  \Var_{\gamma_{x,1}}(g) \leq c \int_{\dR} \frac{1}{1+(v-x)^2}
  g'(v)^2\, d\gamma_{x,1}(v)\,.
  \]
  Since $1+v^2 \leq 1+2(v-x)^2 + 2x^2 \leq 2(1+x^2) (1 + (x-v)^2)$,
  the claim holds for the Gaussian measure~$\gamma_{x,1}$ in
  dimension~$1$.

  To prove the general case, just remark that, for any $x\in \dR^d$,
  $\gamma_{x,1}$ is the product of the (one dimensional) measures
  $\gamma_{1,x_i}$.  The classical tensorization property for
  Poincaré--type inequalities yields
  \[
  \Var_{\gamma_{x,1}}(g) \leq 2c \max_{i}(1+x_i^2) \int_{\dR}
  \sum_{i=1}^d \frac{1}{1+v_i^2} (\partial_i g(v))^2\,
  d\gamma_{x,1}(v)\,,
  \]
  which completes the proof.
\end{proof}

This result extends to mixture of Gaussian measures.
\begin{prop}[Weighted Poincaré inequality for $\mu\star\gamma_1$]
  Let $\mu$ be a probability measure on $\dR^d$ supported in
  $B_d(0,R)$. The probability $\mu \star \gamma_\delta$ satisfies the
  following weighted Poincaré inequality: for all $\mathcal{C}^1$
  function $f$ on $\dR^d$,
  \begin{equation}
    \label{eq:wPoinc}
    \Var_{\mu\star \gamma_1} (f) 
    \leq C(R) \int_{\dR^d} \sum_{i=1}^d \frac{1}{1+u_i^2} 
    (\partial_i f(u))^2\,d(\mu\star\gamma_1)(u)\,,
  \end{equation}
  with $C(R) =c (1+R^2)e^{4R^2}$ for some universal constant $c$.
\end{prop}
\begin{proof}
  According to Lemma~\ref{lem:wP-Gauss}, for all $x \in \dR^d$ such
  that $|x|\leq R$, it holds
  \[
  \Var_{\gamma_{x,1}}(f) \leq c (1+R^2)\int_{\dR} \sum_{i=1}^d
  \frac{1}{1+u_i^2} (\partial_i f(u))^2\, d\gamma_{x,1}(u)
  \]
  for all $\mathcal{C}^1$ function $f$ on $\dR^d$. Inserting these
  weighted Poincaré inequalities into the proof given in
  Section~\ref{sec:Poincare} immediately yields the desired bound.
\end{proof}

We now arrive at a first transportation-entropy inequality.
\begin{thm}
  \label{thm:complicatedTransport}
  Let $\mu$ be a probability measure on $\dR^d$ having its support in
  $B_d(0,R)$. The probability $\mu \star \gamma_1$ satisfies the
  following transport-entropy inequality: for any probability measure
  $\nu$ on $\dR^d$,
  \[
  \cT_k(\nu,\mu\star\gamma_1) \leq c'(1+R^2)\exp(4R^2) H(\nu |\mu\star
  \gamma_1)\,,
  \]
  where $c'$ is a universal constant and $\cT_k$ is the optimal
  transport cost related to the cost function
  \[
  k(x,y) = \min\PAR{|x-y|^2 ; |x-y|} + \min\PAR{\|x-y\|_4^4,
    \|x-y\|_4^2},\qquad \forall x,y \in \dR^d\,.
  \]
\end{thm}
Before proving this result, let us show how to deduce
Theorem~\ref{thm:transport} as a corollary. The Euclidean and $\ell^4$
norms on $\dR^d$ satisfy:
\[
\forall z\in \dR^d, \quad \|z\|_4 \leq \ABS{z} \leq d^{1/4} \|z\|_4\,.
\] 
This gives the following lower bound on the cost $k$:
\begin{align*}
  k(x,y) &=    \min\PAR{|x-y|^2 ; |x-y|} + \min\PAR{\|x-y\|_4^4, \|x-y\|_4^2} \\
         &\geq \min\PAR{\|x-y\|_4^2 ; \|x-y\|_4}  + \min\PAR{\|x-y\|_4^4, \|x-y\|_4^2} \\
         &\geq \|x-y\|_4^2.
\end{align*}
By Theorem~\ref{thm:complicatedTransport} we get
\begin{align*}
  \cT_{2,4}(\nu,\mu\star\gamma_1) 
  \leq \cT_k(\nu,\mu\star\gamma_1) 
  \leq c'\PAR{1+R^2}\exp\PAR{4R^2} H(\nu | \mu\star\gamma_1); 
\end{align*}
where we recall that $\cT_{2,4}$ is the transportation cost associated
to $(x,y)\mapsto \|x-y\|^4$.  The inequality for a general $\delta$
follows by a simple scaling argument.  The inequality for the
Euclidean cost $\cT_2$ is proved in the same way, by bounding $k(x,y)$
from below by $d^{-1/2} \ABS{x-y}^2$.  This concludes the proof of
Theorem~\ref{thm:transport}.

\begin{proof}[Proof of Theorem~\ref{thm:complicatedTransport}]
We proceed in two steps. 

\paragraph{1. A transport-entropy inequality with an intricate cost.}
Let us define three functions $\alpha$, $\omega$ and $T$ by
\begin{align*}
  \forall u \in \dR, \quad 
  \omega(u) &= \mathrm{sign} (u) \left(\ABS{u}+\frac{u^2}{2}\right)\,; \\
  \forall u \in \dR, \quad 
  \alpha(u) &= \min(u^2 ; |u|)\,; \\
  \forall x\in \dR^d, \quad 
  T(x) &= (\omega(x_1),\ldots,\omega(x_d))\,. 
\end{align*}

According to \cite[Theorem 4.6]{Goz10}, the weighted Poincaré
inequality \eqref{eq:wPoinc} implies (and is actually equivalent to)
the following transport cost inequality: for all probability measure
$\nu$ on $\dR^d$,
\[
\cT_{\tilde{k}}(\nu,\mu\star \gamma_1) \leq H(\nu|\mu\star
\gamma_\delta),
\]
where the cost function $\tilde{k}$ is defined by
\begin{equation}\label{eq:cost}
  \tilde{k}(x,y) = \alpha \left(\frac{1}{D} \ABS{T(x) - T(y)} \right)\,,
  \qquad \forall x,y \in \dR^d
\end{equation}
and where $D = c'' \sqrt{C(R)}$ for some universal constant $c''$.

For the sake of completeness, let us give the short proof of the
implication we need.  Let us begin by showing that the measure
$\tilde{\mu} := T_\# (\mu \star \gamma_{1})$ satisfies the usual
Poincaré inequality with the constant $2C(R).$ Indeed, if $f$ is a
$\mathcal{C}^1$ function, applying the weighted Poincaré
inequality~\eqref{eq:wPoinc} to $g=f\circ T$ and using the elementary
bound $(\omega'(v))^2 \leq 2(1+ v^2)$ yields:
\begin{align*}
  \Var_{\tilde{\mu}}(f) 
  &\leq C(R) \int
    \sum_{i=1}^d \frac{1}{1+v_i^2} \omega'(v_i)^2 (\partial_i f)^2(T(v))\,%
    d(\mu\star \gamma_1)(v) \\
  &\leq 2C(R) \int \ABS{\nabla f}^2(u) d\tilde{\mu}(u). 
\end{align*}
According to a well known result by Bobkov, Gentil and Ledoux
\cite[Corollary 5.1]{BGL01} showing the equivalence between the
Poincaré inequality and a transport inequality involving a
quadratic-linear cost, the probability $\tilde{\mu}$ satisfies the
following: for any probability measure $\nu$ on $\dR^d$,
\[
\cT_{\rho}(\nu,\tilde{\mu}) \leq H(\nu | \tilde{\mu}),
\]
where the cost function $\rho : \dR^d\times \dR^d \to \dR^+$ is
defined by
\[
\rho(x,y) = \alpha\left( \frac{1}{D} \ABS{x-y}\right),\qquad x,y \in
\dR^d\,,
\]
where $D=c''\sqrt{C(R)}$ for some universal constant~$c''$.  Let $\nu$
be a probability measure on $\dR^d$ and let $(\tilde{X},\tilde{Y})$ be
an optimal coupling between $\tilde{\nu}:= T_\# \nu$ and $\tilde{\mu}$
(for the transport cost $\cT_\rho$) and denote by
$X = T^{-1}(\tilde{X})$ and $Y = T^{-1}(\tilde{Y})$. Then $(X,Y)$ is a
coupling between $\nu$ and $\mu\star \gamma_1$ and it holds
\[
  \esp{\tilde{k}(X,Y)} 
   =  \esp{\rho (T(X),T(Y))} 
   = \esp{\rho (\tilde{X} ,\tilde{Y})} 
   = \mathcal{T}_\rho(\tilde{\nu},\tilde{\mu}) \leq H(\tilde{\nu} | \tilde{\mu}) 
   = H(\nu | \mu\star \gamma_1),
\]
where the last equality comes from the fact that if
$\nu \ll \mu\star \gamma_1$, then $\tilde{\nu} \ll \tilde{\mu}$ with
\[
  \frac{d\tilde{\nu}}{d\tilde{\mu}}(u) 
  = \frac{d\nu}{d(\mu\star \gamma_1)} \PAR{T^{-1}(u)},
  \qquad \forall u\in \dR^d.
\]
This concludes the first step.

\paragraph{A lower bound on the cost function $\tilde{k}$.}

We now bound $\tilde{k}(x,y)$ from below by the more convenient cost
function $k(x,y)$.  According to \cite[Lemma 2.6]{Goz10},
$|\omega(u)-\omega(v)| \geq \omega (|u-v|/2)$, for all $u,v \in
\dR$. Therefore, for all $x$, $y$ in $\dR^d$:
\begin{align*}
  \ABS{T(x) - T(y)}^2 
  &= \sum_i \ABS{\omega(x_i) - \omega(y_i)}^2 \\
  &\geq \sum_i \omega\PAR{ \frac{\ABS{x_i - y_i}}{2}}^2 \\
  &= \sum_i  \PAR{ \frac{1}{2} \ABS{x_i - y_i} + \frac{1}{8} \ABS{x_i - y_i}^2}^2 \\
  &\geq \frac{1}{4} \sum_i \ABS{x_i - y_i}^2 + \frac{1}{64} \sum_i \ABS{x_i - y_i}^4 \\
  &\geq \frac{1}{32}\PAR{ \frac{1}{2} \ABS{x-y}^2 + \frac{1}{2} \|x-y\|_4^4}.
\end{align*}

Using the inequality $\alpha(au) \geq \alpha(a)\alpha(u)$ for all
$a,u \in \dR$ (\cite[Lemma 2.6]{Goz10}) and the concavity of the
function $u\mapsto \alpha(\sqrt{u})$, $u\in \dR^+$, this leads to the
following bound on the cost function $\tilde{k}$:
\begin{align*}
  \tilde{k}(x,y)  
  &\geq \alpha\PAR{ \frac{1}{D\sqrt{32}} 
    \PAR{ \frac{1}{2} \ABS{x-y}^2 + \frac{1}{2} \|x-y\|_4^4}^{1/2}} \\
  &\geq \frac{1}{2}\alpha\PAR{\frac{1}{D\sqrt{32}}} \PAR{\alpha( |x-y|) 
    + \alpha(\|x-y\|_4^2)}\,,
\end{align*}
Finally, it is easy to check that
$\alpha\PAR{\frac{1}{D\sqrt{32}}} \geq \frac{c'''}{C(R)} $ for some
universal constant $c'''$, which completes the proof.
\end{proof}

\begin{rem}
  If one could improve the conclusion in the result by
  Bobkov, Gentil, Ledoux and conclude that $\tilde{\mu}$ satisfies the
  transport inequality with the cost function
  \[
  (x,y) \mapsto \sum_{i=1}^d \alpha\left( \frac{1}{D}
    \ABS{x_i-y_i}\right)
  \]
  instead of $\rho$, then one would conclude that $\mu$ satisfies
  Talagrand's inequality, with respect to the Euclidean norm, with a
  dimension free constant.
\end{rem}

\section{Special cases and extensions}
\label{sec:partial}
\subsection{Spherically symmetric measures}

We prove in this section the following claim of
Theorem~\ref{thm:partial}:
\begin{thm}
  \label{thm:symmetry}
  If $\mu$ is a spherically symmetric measure with support in
  $B_d(0,R)$, then $\mu\star\gamma_\delta$ satisfies a logarithmic
  Sobolev inequality and
  \[
  C_{LS}(\mu \star \gamma_\delta) \leq 4\delta^2\exp\PAR{\frac{8}{\pi}
    \frac{R^2}{\delta^2}}.
  \]
\end{thm}
Let us recall that $\mu\star\gamma_\delta$ is the law of the random
variable $S=X+\delta Z$. By assumption, the law $\mu$ of $X$ is
spherically symmetric, that is, invariant by any vectorial rotation of
$\dR^d$. Since $Z$ has the same invariance, this implies that the
density $p(z)$ of $S$ only depends on the norm of $z$, thus we can
write:
\[
p(z)=p(|z|e_1) =\int_{\dR^d}\! \frac{1}{(2\pi
  \delta^2)^{d/2}}\exp\PAR{-\frac{1}{2\delta^2}
  \PAR{\PAR{|z|-x_1}^2+\sum_{i=2}^dx_i^2}} d\mu(x_1,x_2,\ldots,x_d)\,.
\]
Denoting, for all $r\in\dR$,
\[
\hat{p}_\delta(r)=\int_{\dR}\! \frac{1}{(2\pi
  \delta^2)^{1/2}}\exp\PAR{-\frac{\PAR{|z|-x_1}^2}{2\delta^2}}
d\hat\mu_1(x_1)
\]
the density of the convolution of $\gamma_\delta$ with the first
marginal $\hat\mu_1$ of the measure
\[
\frac{1}{(2\pi
  \delta^2)^{(d-1)/2}}\exp\PAR{-\frac{1}{2\delta^2}\sum_{i=2}^dx_i^2}
d\mu(x_1,x_2,\ldots,x_d)\,,
\]
one has $p(z)=\hat{p}_\delta(|z|)$.

Since the one-dimensional measure $\hat{\mu}_1$ is supported in the
interval $[-R,R]$, the method from Section~\ref{sec:Miclo} apply.
Using Lemma~\ref{lem:Royer}, with $\sigma=2Ra_1$, we obtain a
decomposition
\[
-\log(\hat{p}_\delta(r))=w_\sigma(r)+w_b(r)\,,
\]
where $w_\sigma\,:\,\dR\rightarrow\dR$ is $1/(2\delta^2)$-convex and
$w_b\,:\,\dR\rightarrow\dR$ is bounded by $2(Ra_1/\delta)^2$.

Since the measure $\hat{\mu}_1$ is symmetric, the function
$\hat{p}_\delta$ is even, so that $w_\sigma$ and $w_b$ constructed in
the proof of Lemma \ref{lem:Royer} are even too.

This entails a decomposition of $p$ on $\dR^d$ as a sum
\[
-\log(p(z))=W_\sigma(z)+W_b(r)
\]
by taking $W_\sigma(z)=w_\sigma(|z|)$ and $W_b(z)=w_b(|z|)$.  The
function $W_b$ is of course bounded by $2(Ra_1/\delta)^2$.  We prove
in Lemma~\ref{lem:convexite} below that $W_c$ is convex.  The
conclusion follows by the same reasoning as in
Section~\ref{sec:Miclo}.

\begin{lem}
  \label{lem:convexite}
  Let $w:\,\dR\rightarrow\dR$ be a $\cC^2$, even, and $\rho$-convex
  function. Then $W:\,\dR^d\rightarrow\dR$ defined by $W(z)=w(|z|)$
  for all $z\in\dR^d$ is also $\cC^2$ and $\rho$-convex.
\end{lem}

\begin{proof}
  Let us denote $N(z)=|z|$. For any $z\neq0$, one computes
  \begin{align*}
    \nabla N(z)&=\frac1{|z|}z\\
    \Hess N(z)&=\frac1{|z|}\PAR{I_d-\frac1{|z|^2}zz^T} 
  \end{align*}
  By composition with $w$, one deduces, for any $z\neq 0$,
  \begin{align*}
    \nabla W(z)&=\frac{w'(|z|)}{|z|}z\\
    \Hess W(z)&=\frac{w''(|z|)}{|z|^2}zz^T+\frac{w'(|z|)}{|z|}
                \PAR{I_d-\frac1{|z|^2}zz^T}\,. 
  \end{align*}
  These two quantities converge respectively to $0$ and $w''(0)I_d$
  when $z\to 0$. By a classical continuation lemma, this implies that
  $W$ is $\cC^2$ with $\nabla W(0) = 0$ and $\Hess W(0) = w''(0)I_d$.

  By assumption, $w''(|z|)\geq \rho$ for any $z\in\dR^d$. Furthermore,
  for any $z\neq0$, $\frac{w'(|z|)}{|z|}\geq \rho$ (since the
  assumptions imply that $0$ is a minimum of $w$).  Finally, noting
  that $zz^T$ and $\PAR{I_d-\frac1{|z|^2}zz^T}$ are the orthogonal
  projections on $\Vect(z)$ and $z^\perp$, one gets that
  $\Hess W(z)\geq \rho I_d$ for any $z\in\dR^d$.
\end{proof}

\subsection{Dimension free log-Sobolev for
  \texorpdfstring{$\delta \in (R/\sqrt{2},R)$}{R/sqrt(2)<delta<R}}
The first item of Theorem~\ref{thm:logSob} states that for $\delta>R$,
the probability measure $\mu\star \gamma_\delta$ satisfies a
logarithmic Sobolev inequality with an explicit, dimension free,
constant. In this section, we improve on this result by proving the
first point of Theorem~\ref{thm:partial}.

The proof of the following result relies on the connections between
functional inequalities and concentration of measure inequalities. The
well known Herbst argument shows that the logarithmic Sobolev
inequality implies a Gaussian concentration of measure
phenomenon. More precisely, if $\mu$ is a probability measure on
$\dR^d$ satisfying the logarithmic Sobolev inequality with a constant
$C_{LS}$, then for any $1$-Lipschitz function $f:\dR^d \to \dR$, it
holds
\[
\mu \left( f \geq m + t\right) \leq e^{-t^2/{C_{LS}}},\qquad \forall t
\geq0,
\]
where $m = \int f\,d\mu$ (see \textit{e.g.} Theorem 5.3 of
\cite{Ledoux-book}).  On the other hand, a recent result by E. Milman
\cite{Mil10} shows that conversely under some curvature assumptions a
sufficiently strong Gaussian concentration of measure inequality
implies back the logarithmic Sobolev inequality. It appears that in
the range of parameters $R /\sqrt{2}<\delta<R$ the measure
$\mu\star \gamma_\delta$ is sufficiently concentrated to apply
Milman's result.

\begin{thm}
  \label{thm:deltaMoyen}
  Suppose that $R /\sqrt{2}<\delta<R$, then $\mu\star \gamma_\delta$ satisfies
  a logarithmic Sobolev inequality with a constant depending only on $R$ and
  $\delta$ and not on $d$. 
\end{thm}

\begin{proof}
  Let us examine the concentration properties of $X+\delta Z$ where
  $X$ and $Z$ are independent random variables with respective laws
  $\mu$ and $\cN_d(0,I_d)$.  If $f : \dR^d \to \dR$ is a $1$-Lipschitz
  function, then denoting by $m = \esp{f(X+\delta Z)}$, it holds for
  any $t \geq0$
  \begin{align*}
    \prb{ f(X + \delta Z) \geq m + t}  
    & = \esp[X]{  \prb{ f(X + \delta Z) \geq m + t \,|\, X}} \\
    & \leq \esp[X]{ \exp\left(-\frac{1}{2\delta^2} 
      \left[t+m-\esp[Z]{f(X+\delta Z)}\right]_+^2\right)}
  \end{align*}
  where the second inequality follows from the concentration
  inequality satisfied by $\delta Z$ (which is for instance a
  consequence of the fact that $\gamma_\delta$ satisfies the
  logarithmic Sobolev inequality with the constant $2\delta^2$).  Now,
  for any $x\in B_d(0,R)$,
  \[
  \ABS{m - \esp[Z]{f(x+\delta Z)}} = \left| \esp[X]{ \esp[Z]{ f(X+\delta
        Z) - f(x+ \delta Z) }} \right| \leq \esp[X]{ \ABS{X-x}} \leq
  2R.\]
  Therefore $\esp[Z]{f(X+\delta Z)} \leq 2R + m$ almost surely, hence
  \[
  \prb{ f(X + \delta Z) \geq m + t} \leq
  \exp\left(-\frac{1}{2\delta^2} \left[t-2R\right]_+^2\right).
  \]
  In particular, for any $0<\eps<1$, it holds
  \[
  \prb{f(X + \delta Z) \geq m + t} \leq
  \exp\left(-\frac{\eps}{2\delta^2}t^2 \right), \qquad \forall t >
  \frac{2R}{1-\sqrt{\eps}} := t_\eps.
  \]

  On the other hand, the density of the law of $X+\delta Z$ is of the
  form $e^{-V_\delta}$, with a function $V_\delta$ such that
  $\mathrm{Hess}\, V_{\delta} \geq \frac{1}{\delta^2} -
  \frac{R^2}{\delta^4} = - \kappa_\delta$.
  In this range of parameters, $\kappa_\delta >0$. According to
  Theorem 1.2 of \cite{Mil10}, as soon as
  $\frac{\eps}{2\delta^2} \geq \frac{1}{2} \kappa_\delta$ (which means
  that $R/\delta <\sqrt{1+\eps}$), the probability measure $\mu$
  satisfies a Gaussian isoperimetric inequality, which in turn implies
  the logarithmic Sobolev inequality with a constant depending only on
  the parameters $\eps, R, \delta$.
\end{proof}

\subsection{Dimension free log-Sobolev for log-convex functions}
\label{sec:Maurey}
Recall the following results by Maurey.
\begin{thm}[{\cite[Theorem 3]{Ma91}}]
  Let $X$ be a bounded random variable such that $\ABS{X} \leq R$
  a.s. Then $X$ satisfies the so called convex $\tau$-property :
  \[
  \esp{e^{Q_{4R^2}f(X)}} \esp{e^{-f(X)}} \leq 1,
  \]
  for any convex function $f : \dR^d \to \dR$, where
  $Q_sf(x) = \inf_{y\in \dR^d}\BRA{ f(y) + \frac{\ABS{x-y}^2}{4s}   }$,
  $s>0.$
\end{thm}
On the other hand, the Gaussian random variable $\delta Z$ with law
$\cN_d(0,\delta I_d)$ satisfies the following $\tau$-property
\[
\esp{e^{Q_{\delta^2} f (\delta Z)}} \esp{e^{-f(\delta Z)}} \leq 1,
\]
for any function $f: \dR^d \to \dR$ (\cite[Theorem 2]{Ma91}).

By the tensorization property of the convex $\tau$-property
(\cite{Ma91}), one concludes that $(X,\delta Z)$ satisfies the
following $\tau$-property
\[
\esp{e^{\tilde{Q}f (X,\delta Z)}} \esp{e^{-f(X,\delta Z)}} \leq 1,
\]
for any convex function $f : \dR^d\times \dR^d \to \dR$, where
\[
\tilde{Q}f(x_1,x_2) = \inf_{(y_1,y_2) \in \dR^d \times \dR^d} 
\left\{f(y_1,y_2) + \frac{1}{16R^2} \ABS{x_1-y_1}^2 + 
\frac{1}{4\delta^2} \ABS{x_2-y_2}^2\right\}.
\]
In particular, applying the inequality above to
$f(x_1,x_2) = g(x_1+x_2)$, and using the fact that
\[
\inf_{y_1+y_2 =y} \left\{ \frac{1}{16R^2} \ABS{x_1-y_1}^2 +
  \frac{1}{4\delta^2} \ABS{x_2-y_2}^2 \right\} =
\frac{1}{4C(\delta,R)}\ABS{x_1+x_2 - y}^2,
\]
with $C(\delta,R) = \delta^2 + 4R^2$, one concludes that $X+ \delta Z$
satisfies
\[
\esp{e^{Q_C g (X+\delta Z)}}\esp{e^{-g(X+\delta Z)}} \leq 1,
\]
for any convex function $g: \dR^d \to \dR.$

According to \cite{GRST14}, this inequality is equivalent to the
following transport type inequality
\[
\overline{\mathcal{T}}_2(\nu_1,\nu_2) \leq C(\delta,R)
\left(H(\nu_1|\mu\star\gamma_\delta) +
  H(\nu_2|\mu\star\gamma_\delta)\right),
\]
for all probability measures $\nu_1,\nu_2$ on $\dR^d$, where
$H(\,\cdot\,|\mu\star\gamma_\delta)$ denotes the relative entropy
functional and
\[
\overline{\mathcal{T}}_2(\nu_1,\nu_2) = \inf_{X_1 \sim \nu_1,\ X_2
  \sim \nu_2} \esp{\ |X_1 -\esp{X_2 |X_1}|^2 \ }.
\]
It is also shown in \cite{GRST14} that this transport inequality
implies the following logarithmic Sobolev inequality
\[
\mathrm{Ent}_{\mu\star\gamma_\delta} (e^f) \leq 8(\delta^2 + 4 R^2)
\int |\nabla f |^2 e^f \,d\mu\star\gamma_\delta,
\]
for any \emph{convex} function $f: \dR^d \to\dR$. This proves the fourth
item of Theorem~\ref{thm:partial}.

\printbibliography

\bigskip


{\footnotesize %
  \noindent Jean-Baptiste~\textsc{Bardet}, 
  e-mail: \texttt{jean-baptiste.bardet(AT)univ-rouen.fr}

  \medskip

  \noindent\textsc{
  LMRS, Université de Rouen,
Avenue de l'Université, BP 12,
Technopôle du Madrillet,
76801 Saint-Étienne-du-Rouvray,
France.}

\bigskip

  \noindent Nathaël~\textsc{Gozlan},
e-mail: \texttt{natael.gozlan(AT)u-pem.fr}

 \medskip

 \noindent\textsc{LAMA UMR 8050, CNRS-Université-Paris-Est-Marne-La-Vallée, 
 5, boulevard Descartes,
 Cité Descartes, Champs-sur-Marne,
 77454 Marne-la-Vallée Cedex 2, France.}

\bigskip

 \noindent Florent \textsc{Malrieu},
 e-mail: \texttt{florent.malrieu(AT)univ-tours.fr}

 \medskip

 \noindent\textsc{
Laboratoire de Mathématiques et Physique Théorique
(UMR CNRS 7350), Fédération Denis Poisson (FR CNRS 2964), Université
François-Rabelais, Parc de Grandmont, 37200 Tours, France.
}

\bigskip

   \noindent Pierre-Andr\'e~\textsc{Zitt},
e-mail: \texttt{pierre-andre.zitt(AT)u-pem.fr}

 \noindent\textsc{LAMA UMR 8050, CNRS-Université-Paris-Est-Marne-La-Vallée, 
 5, boulevard Descartes,
 Cité Descartes, Champs-sur-Marne,
 77454 Marne-la-Vallée Cedex 2, France.}

}

\end{document}
